\documentclass[10pt]{amsart}
\usepackage{amssymb,amsfonts}
\usepackage[all,arc]{xy}
\usepackage{enumitem}
\usepackage{mathrsfs}
\usepackage{ stmaryrd }
\usepackage{url}

\usepackage
[pdfauthor={Carl Wang-Erickson},
 pdftitle={Higher Yoneda product structures and Iwasawa algebras modulo p},
 bookmarks=false]
{hyperref}

\setlist[enumerate]{leftmargin=2em}
\setlist[itemize]{leftmargin=2em}

\newcommand\define{\newcommand}

\newcommand{\Z}{\mathbb{Z}}

\newcommand{\R}{\mathbb{R}}

\newcommand{\frg}{\mathfrak{g}}
\newcommand{\frk}{\mathfrak{k}}

\newcommand{\Hom}{\mathrm{Hom}}

\newcommand{\Mod}{\mathrm{Mod}}
\newcommand{\Ext}{\mathrm{Ext}}

\newcommand{\End}{\mathrm{End}}

\newcommand{\ind}{\mathrm{ind}}

\newcommand{\lb}{{[\![}}
\newcommand{\rb}{{]\!]}}

\define\cA{\mathcal{A}}

\define{\Fitt}{\mathrm{Fitt}}
\define{\Ann}{\mathrm{Ann}}

\newtheorem{thm}{Theorem}[subsection] 
\newtheorem*{thm*}{Theorem}

\newtheorem{cor}[thm]{Corollary}
\newtheorem{prop}[thm]{Proposition}
\newtheorem{lem}[thm]{Lemma}

\newtheorem{ques}[thm]{Question}

\theoremstyle{definition}
\newtheorem{defn}[thm]{Definition}

\newtheorem{notn}[thm]{Notation}

\theoremstyle{remark}
\newtheorem{rem}[thm]{Remark}

\newcommand{\ra}{\rightarrow}

\newcommand{\lra}{\longrightarrow}
\newcommand{\lrisom}{\buildrel\sim\over\lra}
\newcommand{\risom}{\buildrel\sim\over\ra}

\newcommand{\rsurj}{\twoheadrightarrow}

\newcommand{\cH}{\mathcal{H}}

\usepackage{color}

\newcommand{\B}{\mathrm{Bar}}
\newcommand{\op}{\mathrm{op}}

\makeatletter
\let\c@equation\c@thm
\makeatother
\numberwithin{equation}{subsection}

\author{Carl Wang-Erickson}

\address{Department of Mathematics \\ University of Pittsburgh  \\ Pittsburgh, PA 15260, USA}

  \email{carl.wang-erickson@pitt.edu}
   \urladdr{https://sites.pitt.edu/\textasciitilde caw203/}

\begin{document}

\title[Higher products and Iwasawa algebras]{Higher Yoneda product structures and Iwasawa algebras modulo $p$}

\begin{abstract}
We give answers to three questions posed by Sorensen. These concern the relationship between the modulo $p$ Iwasawa algebra of a torsionfree pro-$p$ group and $A_\infty$-algebra structures on its Yoneda algebra. 
\end{abstract}

\maketitle

\tableofcontents

\section{Introduction}

Let $p$ be a prime number and let $k$ be a finite field of characteristic $p$. The subject of this paper is the $k$-linear representation theory of $p$-adic Lie groups $G$. This is of natural interest relative to the proposed $p$-adic local Langlands correspondence. For an introduction to this connection, see \cite{harris2016}. 

\subsection{Derived equivalences of Schneider and of Sorensen} 
\label{subsec: intro equivalences} 

As Schneider pointed out \cite{schneider2015}, the passage from $k$-linear smooth representations of $G$ over $k$ to modules for various Hecke algebras is not exact, losing some information. More specifically, we let $I \subset G$ denote a compact open subgroup that is torsionfree and pro-$p$ and define the Hecke algebra
\[
\cH_I := \End_{\Mod^\mathrm{sm}_k(G)}(\ind_I^G(k))^\op,
\]
so that the usual passage is
\[
H^0 : \Mod_k^\mathrm{sm}(G) \lra \Mod(\cH_I), \quad V \mapsto V^I. 
\]
Schneider proposed a derived framework that will not lose information. It is based on the (differential graded) dg-Hecke algebra 
\[
\cH_I^\bullet = \End_{\Mod^\mathrm{sm}_k(G)}(J^\bullet_I)^\op, 
\]
where $J^\bullet_I$ is an injective resolution of $\ind_I^G(k)$. Let $D(-)$ denote the formation of a derived category. Schneider showed that the passage 
\[
H^\bullet : D(\Mod_k^\mathrm{sm}(G)) \to D(\Mod_\mathrm{dg}(\cH_I^\bullet))
\]
is an equivalence of triangulated categories \cite[Thm.\ 9]{schneider2015}. This $H^\bullet$ is a natural derived enrichment of $H^0$, sending complexes $V^\bullet$ of injective objects of $\Mod_k^\mathrm{sm}(G)$ to $\Hom_{\Mod_k^\mathrm{sm}(G)}^\bullet(J_I^\bullet, V^\bullet)$. 

Applying results of Kadeishvili and Lef\`{e}vre-Hasegawa, Sorensen \cite{sorensen2020} exhibited a further equivalence from $D(\Mod_\mathrm{dg}(\cH_I^\bullet))$ to a category $D_\infty((\Omega^!, m))$ of $A_\infty$-modules for an $A_\infty$-algebra structure $m$ on the Yoneda algebra 
\[
\Omega^! := \Ext^*_{\Mod^\mathrm{sm}_k(G)}(\ind_I^G(k), \ind_I^G(k))
\]
that could be called the ``derived Hecke algebra,'' although the exposition of \cite{sorensen2020} emphasizes the case $G=I$. This equivalence is the result of passing from cochains and dg-algebras to the graded algebra of cohomology; for example, there is an isomorphism $H^*(\cH_I^\bullet) = \Omega^!$, but the canonical graded algebra structure on $\Omega^!$, known as the Yoneda algebra, loses information from $\cH_I^\bullet$ in general. Kadeishvili produced  an $A_\infty$-algebra structure $m$ on $\Omega^!$ along with a quasi-isomorphism of $A_\infty$-algebras between $\cH_I^\bullet$ and $(\Omega^!,m)$ \cite{kadeishvili1982}, and Lef\`{e}vre-Hasegawa produced an accompanying equivalence of derived categories of modules \cite{LH2003}. 

\subsection{Main result}
In this paper, we contribute to the equivalences Sorensen studied, ``bringing them full circle'' in a certain sense that we will now explain. Like Sorensen, we will emphasize the case that $G=I$, due the difficulty in computing $\Omega^!$ (merely as a graded $k$-vector space) in all but a handful of cases. We will let $\Omega := k\lb G\rb$ denote the Iwasawa algebra; then we may apply the equivalence produced by Pontryagin duality, 
\[
\Mod(\Omega)^\op \lrisom \Mod_k^\mathrm{sm}(G), 
\]
where $\Mod(\Omega)$ denotes the category of pseudocompact left $\Omega$-modules. The series of equivalences of derived categories that we want to ``bring full circle'' is
\begin{equation}
\label{eq: functor composition}
D(\Mod(\Omega)^\op) \risom D(\Mod_k^\mathrm{sm}(G)) \risom D(\Mod_\mathrm{dg}(\cH_G^\bullet)) \risom D_\infty((\Omega^!, m)),
\end{equation}
as seen in \cite[pg.\ 153]{sorensen2020}. 
 
Our main result is 
\begin{thm}[{Corollary \ref{cor: str transfer}}]
\label{thm: main}
Assume that the $p$-adic Lie group $G$ is pro-$p$. The Iwasawa algebra $\Omega = k\lb G \rb$ can be reconstructed, up to isomorphism, from $(\Omega^!,m)$ as the classical hull of the dual bar construction of its opposite $A_\infty$-algebra. Moreover, a choice of homotopy retract between $\cH_G^\bullet$ and $\Omega^!$ naturally pins down an isomorphism from $\Omega$ to this hull.
\end{thm}

The notions of \emph{dual bar construction} and \emph{classical hull} are explained in \S\ref{subsec: bar} and \S\ref{subsec: classical hull}, respectively. In brief, the dual bar construction translates the information of $(\Omega^!, m)$ into a complete dg-algebra augmented over $k$, and its classical is the universal (classical) augmented $k$-algebra receiving a map of $k$-augmented dg-algebras from the dual bar construction. 

We remark that the assumption of the theorem is looser than that of \cite{sorensen2020}, which also assumes that $G$ is torsionfree. The torsionfree assumption is needed to make the derived categories manageable, but it is not needed for this reconstruction of $\Omega$. 

When we restore the torsionfree assumption so that the whole chain of equivalences \eqref{eq: functor composition} is valid, Theorem \ref{thm: main} allows us to pass from the target of \eqref{eq: functor composition} back to its source in a particularly strong sense, reconstructing $\Omega$, and hence the abelian category $\Mod(\Omega)$, from $D_\infty((\Omega^!, m))$. Indeed, it is possible to reconstruct the $A_\infty$-algebra $(\Omega^!,m)$, up to isomorphism, from its module category $D_\infty((\Omega^!,m))$, according to \cite[Thm.\ 7.6.0.6]{LH2003} (cf.\ \cite[Thm.\ 3.1]{keller2006}).

\subsection{Positive answers to Sorensen's questions} 
This paper emphasizes the application of the main theorem to answer questions of Sorensen. Here are Sorensen's questions, reproduced verbatim modulo the following changes. 
\begin{itemize}
\item $K \subset G$ denotes a subgroup, instead of $H \subset G$
\item Following this paper's convention of writing $(\Omega^!, m)$ for an $A_\infty$-algebra structure on $\Omega^!$ that extends its inherent graded algebra structure, the notation $\Omega^!$ is left to denote the underlying dg-algebra (with trivial differential). 
\end{itemize}

Sorensen uses an explicit expression for the Yoneda graded algebra $\Omega^!$ as an exterior algebra on a vector space $\frg^*$, upon an additional assumption on $G$ that we explain in \S \ref{subsec: sorensen setting}, in order to express questions (b) and (c). 

\begin{ques}[{Sorensen \cite[\S12]{sorensen2020}}] 
\label{ques: main} 
Here are the questions.
\begin{enumerate}[label=(\alph*)]
\item \textit{By \cite[Thm.\ 1.1]{sorensen2020}, one can recover $\Omega = k\lb G\rb$ up to derived equivalence from the $A_\infty$-algebra
$(\Omega^!, m) = (\Ext^\bullet_\Omega(k, k)^\mathrm{op}, m)$. Does $(\Omega^!,m)$ determine $\Omega$ up to isomorphism?}
\item \textit{Is there a converse to \cite[Thm.\ 1.2]{sorensen2020} in the sense that $G$ must be abelian if the $A_\infty$-structure on $\bigwedge \frg^*$ is trivial?} 
\item \textit{Suppose $K \subset G$ is an open subgroup. Then $\Omega(G)$ is finite free over the subalgebra $\Omega(K)$ and we have the restriction map $\mathrm{Mod}(\Omega(G)) \ra \mathrm{Mod}(\Omega(K))$ which induces a map $D(\Omega(G)) \ra D(\Omega(K))$. Is there a morphism of $A_\infty$-algebras $(\bigwedge \frg^*,m_G) \ra (\bigwedge \frk^*,m_K)$ inducing the corresponding map on $D_\infty$ via ``extension of scalars'' along this map?}
\end{enumerate}
\end{ques}

Theorem \ref{thm: main} addresses question (a) directly and positively.  It can also be used to address the remaining questions. 
\begin{cor}
\label{cor: main}
Questions (a)-(c) of \ref{ques: main} all have positive answers. 
\end{cor}

\begin{proof}
Theorem \ref{thm: a} provides a positive answer to question (a). Theorem \ref{thm: b} provides a positive answer to question (b). Theorem \ref{thm: c} provides a positive answer to question (c). 
\end{proof}

\subsection{Summary of methods} 
Context for Sorensen's work appears in \S\ref{sec: sorensen}. Next, background from \cite{A-inf} is reproduced in \S\ref{sec: recall}, which culminates in a variation of Theorem \ref{thm: main} that presents $\Omega$ in terms of an $A_\infty$-algebra structure $m'$ enriching the graded algebra structure on the opposite algebra $(\Omega^!)^\op$ (Theorem \ref{thm: A-inf main}). This $m'$ is different than the $A_\infty$-algebra structure $m$ on $\Omega^!$ that has been discussed up until this point; reconciling this difference is the main technical issue dealt with in this paper, which occupies \S\ref{sec: answers}. In order to apply the presentation of $\Omega$ in terms of the $A_\infty$-structure $m'$, the key technical ingredient, furnished by Corollary \ref{cor: str transfer}, is an \emph{explicit isomorphism} of $A_\infty$-algebras between $((\Omega^!)^\op,m^\op)$ and $((\Omega^!)^\op, m')$. The key idea is that while both $m^\op$ and $m'$ arise quite naturally from a single set of choices, they are far from identical; nonetheless, a result of Segal \cite{segal2008}, recorded as Proposition \ref{prop: segal map}, reconciles them. Once this is done, it is relatively straightforward to deduce positive answers to Question \ref{ques: main}.

\subsection{Related works} 
\label{subsec: related}

As discussed in \cite[\S4.5]{A-inf}, the fact that a choice of $A_\infty$-algebra structure $m^\op$ on $(\Omega^!)^\op$ determines a presentation for $\Omega$ was proved by Segal \cite[Thm.\ 2.14]{segal2008} in an analogous situation when $k$ has characteristic zero. (This was also proved in the case of a graded algebra in place of $\Omega$ in \cite{LPWZ2009}.) The extension to general characteristic is given in \cite[Part 2]{A-inf}. In addition, the amplification of \cite{segal2008} given in \cite[Cor.\ 6.2.6]{A-inf} especially clarifies the given answer to question (c). 

L.\ Positselski has previously answered these questions positively, in the sense that positive answers follow from the isomorphism \eqref{eq: Lazard theorem} and rather immediate consequences of his work. Namely, Positselski has studied equivalences of module categories that accompany bar-cobar equivalences of categories of dg-algebras and $A_\infty$-algebras, from which he claims that positive answers can be derived. 
\begin{itemize}
\item A positive answer to question (a) follows from \cite[\S6.10, part (b) of Theorem, pg.\ 76]{positselski2011}. It is also recorded as \cite[Thm.\ 3.3]{positselski2017}. 
\item A positive answer to question (b) may be found in \cite[end of Ex.\ 6.3, pg.\ 225]{positselski2017}. 
\item A positive answer to question (c) follows from \cite[\S6.9, part (a) of Proposition, pg.\ 74]{positselski2011}. 
\end{itemize}

\subsection{Conventions and definitions}
\label{subsec: conventions}

We work with complexes, graded algebras, dg-algebras, and $A_\infty$-algebras over a finite field $k$ of positive characteristic $p$. All gradings in the remainder of this paper are indexed by $\Z$, and the differentials have graded degree $+1$. 

\begin{rem}
The assumption that $p$ is odd is used in \cite[\S2]{sorensen2020} in order to relate the Yoneda algebra $\Omega^!$ to the $k$-Lie algebra $\frg$ via the isomorphism \eqref{eq: Lazard theorem}. This is required only in order to make sense of questions (b) and (c), so we do not make this assumption everywhere. 
\end{rem}

We let $\hat T_k V$ denote the free completed tensor algebra on a graded $k$-vector space $V$. We let $V^*$ denote the graded degree-wise $k$-linear dual of $V$, that is, $(V^*)^n = (V^{-n})^*$. This dual operation extends to complexes. 

We use $\Sigma$ to denote suspension of a (differential) graded $k$-vector space. This is mainly used to move elements of graded vector spaces from degree 1 to degree 0, so that we can consider algebras involving them as (classical) $k$-algebras (as opposed to graded $k$-algebras). 

\begin{rem}
The symbol $\Sigma V^*$ should be read as $(\Sigma V)^*$, first suspending and then applying the graded dual. Very concretely, $(\Sigma V^*)^n = (V^{1-n})^*$. 
\end{rem}

An $A_\infty$-algebra over $k$ is an algebra in graded $k$-vector spaces over the $A_\infty$-operad. In this article, we call these ``$A_\infty$-algebras,'' not mentioning $k$. See the article of Keller \cite{keller2001} for the full definition of the category of $A_\infty$-algebras, matching the convention we use here. Here, we give summary definitions. In particular, when $B, B'$ are graded $k$-vector spaces, we use $(B,m)$, where $m = (m_n)_{n \geq 1}$, to denote an $A_\infty$-algebra structure on $B$, i.e.
\[
m_n : B^{\otimes n} \lra B, \text{ for } n \geq 1, \text{ of graded degree } 2-n
\]
satisfying certain compatibility conditions. 
Likewise, $f = (f_n)_{n \geq 1} : (B,m) \to (B',m')$ denotes a morphism of $A_\infty$-algebras, where 
\[
f_n : B^{\otimes n} \lra B', \text{ for } n \geq 1, \text{ of graded degree } 1-n. 
\]
We also refer to terms describing $A_\infty$-algebras (minimal, formal) and $A_\infty$-morphisms (quasi-isomorphism, etc.) that can be found in \cite{keller2001}. We emphasize that an $A_\infty$-algebra $(B,m)$ is called \emph{minimal} when $m_1 = 0$. 

We will treat dg-algebras $(C, d_C, m_{2,C})$, where $d_C$ is the differential and $m_{2,C}$ is the multiplication, as $A_\infty$-algebras. The $A_\infty$-structure is $m = (m_n)_{n \geq 1}$ where $m_1 = d_C$, $m_2 = m_{2,C}$, and $m_n = 0$ for $n \geq 3$. In contrast, we say that an $A_\infty$-algebra structure $m$ \emph{enriches} a dg-algebra $(C,d_C,m_{2,C})$ when $m_1 = d_C$ and $m_2 = m_{2,C}$; for enrichments $m$, there is no restriction on $m_n$ for $n \geq 3$. We remark that enrichments of graded algebras, considered to be a dg-algebra with a trivial differential, are minimal by definition. For example, the $A_\infty$-algebra enrichments of the Yoneda algebra $\Omega^!$ discussed earlier in this introduction are minimal. 

The works of Keller \cite{keller2001, keller2002, keller2006} are useful introductions to $A_\infty$-algebras in relation to representations of algebras, with respect to the perspective and conventions of this paper.

\section{Sorensen's results}
\label{sec: sorensen}

\subsection{The setting of \cite{sorensen2020}}
\label{subsec: sorensen setting}

We will use some common terminology about $p$-adic Lie groups without giving definitions here, referring the reader to \cite{sorensen2020}, where they are clearly explained. Schneider's book \cite{schneider2011}  is a thorough exposition of this background material. 

Let $G$ be a $p$-adic Lie group that is torsionfree and pro-$p$. Let $k$ be a finite field of characteristic $p$. Let $\Omega = k\lb G\rb$ be the completed group algebra, the \textit{Iwasawa algebra} of $G$, which is a local associative $k$-algebra equipped with its standard profinite topology. Let $D(\Omega)$ denote the derived category of the category $\mathrm{Mod}(\Omega)$ of pseudocompact left $\Omega$-modules, which, as Sorensen explains \cite[\S3]{sorensen2020}, is anti-equivalent to the category of smooth $k$-linear representations of $G$. 
\begin{defn}
\label{defn: Omega shriek}
Let $\Omega^!$ denote the opposite algebra of the Yoneda algebra of $\Omega$, recalling that the Yoneda algebra in the category $\mathrm{Mod}(\Omega)$,
\[
(\Omega^!)^\op = \Ext^\bullet_\Omega(k,k) := {\bigoplus}_{i \in \Z_{\geq 0}} \Ext^i_\Omega(k,k),
\]
is a canonical $\Z$-graded $k$-algebra under the cup product.  We call $\Omega^!$ the \emph{Koszul dual} $k$-algebra of $\Omega$; it plays the role of the derived Hecke algebra, for reasons explained in \cite[\S1]{sorensen2020}. 
\end{defn}

For concreteness, and in order to recall the full scope of Sorensen's results, we set up a narrower class of groups $G$ where the structure of $\Omega^!$ is well-understood (see especially \cite[\S\S7-8]{sorensen2020} and \cite{schneider2011} for reference). When $G$ is equipped with a \emph{valuation}, there arises a graded \emph{$k$-Lie algebra of $G$} that we denote by $\frg$ (see e.g.\ \cite[\S\S23-25]{schneider2011}). When there is a \emph{basis} for $G$ whose elements have the same valuation $t \in \R_{>1/(p-1)}$, $G$ is called \emph{equi-$p$-valued} and $\frg$ is concentrated in degree $t$; in particular, $\frg$ is abelian. Sorensen especially focuses on the case where $G$ is a \emph{uniform} pro-$p$ group, which implies that $p$ is odd, that $G$ is equi-$p$-valued, and that the valuation can be chosen so that $\frg$ is concentrated in degree 1. 

Combining a theorem of Lazard \cite{lazard1965} in the equi-$p$-valued case with consequences of the straightforward nature of $\frg$ in the uniform case, one has a canonical $\Z$-graded $k$-algebra isomorphism of \cite[Cor.\ 8.3]{sorensen2020},
\begin{equation}
\label{eq: Lazard theorem}
\Omega^! \lrisom {\bigwedge}_k \frg^* := \bigoplus_{i \in \Z_{\geq 0}} \wedge_k^i (\frg^*), 
\end{equation}
where $(-)^*$ denotes $k$-linear duality. A particular consequence of \eqref{eq: Lazard theorem} is that $\dim_k \Ext^i_\Omega(k,k) = \binom{\dim_k \frg}{i}$ for integers $i, 0 \leq i \leq \dim_k \frg$. 

This discussion makes it clear that the passage $\Omega \mapsto \Omega^!$ loses information: there are equal-dimensional uniform pro-$p$ groups that are not isomorphic, and thus their Iwasawa algebras are not isomorphic. Yet the $k$-Lie algebras of uniform pro-$p$ groups are determined up to isomorphism by their dimension alone \cite[\S7.1]{sorensen2020}. 

\subsection{The results of \cite{sorensen2020}}  
Sorensen proves that there exists an \emph{$A_\infty$-algebra} structure enriching the graded algebra $\Omega^!$ that recovers the lost information, in the following sense. We denote such a structure by $m$, and write $(\Omega^!, m)$ for the resulting $A_\infty$-algebra. For an introduction to $A_\infty$-algebras, see the references given in \S\ref{subsec: conventions}, where the notion of ``enrichment'' is also discussed. 

We emphasize that $m$ is unique up to non-unique isomorphism, which is typical for $A_\infty$-algebra structures. That is, whilst $m$ is not canonical, the isomorphism class of $(\Omega^!,m)$ is canonical. In particular, $m$ is called \emph{trivial} when it carries no more information than $\Omega^!$; triviality of $(\Omega^!,m)$ is well-defined up to isomorphism. 

There is a derived category of strictly unital left $A_\infty$-modules of $(\Omega^!, m)$, denoted $D_\infty(\Omega^!, m)$. The main result of \cite{sorensen2020} is that there are the following equivalences of triangulated categories via the composition \eqref{eq: functor composition}, 
\[
\tag{\cite[Thm.\ 1.1]{sorensen2020}} D(\Omega) \lrisom D_\infty(\Omega^!, m).  \qquad 
\]
And when $G$ is a uniform pro-$p$ group, this can be rephrased as 
\[
\tag{\cite[Thm.\ 1.2]{sorensen2020}} D(\Omega) \lrisom D_\infty(\bigwedge \frg^*, m).  \qquad 
\]

\section{The reconstruction theorem from \cite{A-inf}}
\label{sec: recall}

In this section, our goal is to state an application of \cite[Cor.\ 6.2.6]{A-inf} to the Iwasawa algebra $\Omega$, which is recorded here as Theorem \ref{thm: A-inf main}. We first recall background that is presented at greater length in \cite[\S5]{A-inf}. 

\subsection{Hochschild cohomology}

Firstly we recall a continuous version of the standard \emph{Hochschild cochain complex} 
\[
C^\bullet(\Omega, k) := \bigoplus_{i \in \Z_{\geq 0}} C^i(\Omega,k) := \bigoplus_{i \in \Z_{\geq 0}} \Hom_k(\Omega^{\otimes i}, k)
\]
of $\Omega$, where the $(\Omega, \Omega)$-bimodule structure on $k$ is trivial and where $\Hom_k(\Omega^{\otimes i}, k)$ consists of those $k$-linear maps that are continuous according to the natural profinite topologies. This is naturally a dg-algebra, where the multiplication comes from the multiplication operation on $k$ and the standard cup product of Hochschild cochains. In what follows, we presume continuity of all Hochschild cochains and pass over topological conditions in silence. 

Likewise, denote the graded $k$-algebra of cohomology of the Hochschild cochain complex, which we will call \emph{Hochschild cohomology}, by $H^\bullet(\Omega, k)$. The Hochschild cohomology $H^\bullet(\Omega, k)$ is canonically isomorphic to the Yoneda algebra $\Ext^\bullet_\Omega(k,k)$, as a particular case of the standard result that, for left $\Omega$-modules that are finite-dimensional over $k$, there is a canonical isomorphism $HH^\bullet(A, \Hom_k(M,N)) \cong \Ext_A^\bullet(M,N)$ \cite[Lem.\ 8.4.2]{witherspoon2019}. In the sequel we will freely use all of the canonical isomorphisms
\begin{equation}
\label{eq: Yoneda algebra} 
(\Omega^!)^\op \cong \Ext^\bullet_\Omega(k,k) \cong H^\bullet(\Omega, k) 
\end{equation}

\subsection{Minimal models for dg-algebras}
\label{subsec: MM}

Let $(C, d_C, m_{2,C})$ be a dg-$k$-algebra with graded cohomology algebra $H = (H^\bullet(C), 0, m_2)$. It is a result of Kadeishvili \cite{kadeishvili1982}, which may also be found recorded \cite[\S2.2]{A-inf}, that there exists an $A_\infty$-algebra structure $m = (m_n)_{n \geq 1}$ on the cohomology of a dg-algebra such that
\begin{itemize}
\item it enriches the graded algebra structure on $H$, in the sense that 
\[
m_1 = 0 \quad \text{ and } \quad m_2 \equiv m_{2,C} \pmod{B^\bullet(C)} 
\]
where $B^\bullet(C)$ represents the graded vector space of coboundaries in $C$. 
\item there exists a quasi-isomorphism of $A_\infty$-algebras 
\[
f = (f_n)_{n \geq 1} : H \ra C
\]
where $f_1$ sends each cohomology class to a choice of representative cocycle. In order to interpret this map in the $A_\infty$ category, recall that dg-algebras may be taken to be $A_\infty$-algebras with trivial higher multiplications, as discussed in \S\ref{subsec: conventions}. 
\end{itemize}
This data $(m,f)$ is unique up to non-unique isomorphism. 

Because an $A_\infty$-algebra $(A,m)$ is called minimal when $m_1 = 0$, we call a $(H,m)$ produced by Kadeishvili a \emph{minimal model} of $(C, d_C, m_{2,C})$, as $f : (H,m) \ra (C, d_C, m_{2,C})$ is a quasi-isomorphism. We call such $(m,f)$ a \emph{minimal model structure} of $H$ relative to $(C, d_C, m_{2,C})$. 

Subsequent work of Kontsevich--Soibelman \cite{KS2000} established the existence of minimal models for $A_\infty$-algebras and clarified that a homotopy retract structure on $(C,d_C)$ relative to $(H,0)$ gives rise to a choice of $(m,f)$ producing the minimal model. 
\begin{defn}
\label{defn: homotopy retract}
Let $(A,d_A)$, $(C,d_C)$ be complexes. We call $(A,d_A)$ a \emph{homotopy retract} of $(C,d_C)$ when they are equipped with maps
\[
\xymatrix{
C \ar@(dl,ul)^h \ar@<1ex>[r]^p & A \ar@<1ex>[l]^i
}
\]
such that $p$ and $i$ are morphisms of complexes, $h : C \ra C[1]$ is a morphism of graded vector spaces, $\mathrm{id}_C - ip = d_Ch + hd_C$, and $i$ is a quasi-isomorphism. 
\end{defn}

\begin{prop}[{Kontsevich--Soibelman \cite{KS2000}}]
\label{prop: KS}
Let $(C,m')$ be an $A_\infty$-algebra. A homotopy retract $(i,p,h)$ between $(H^\bullet(C), 0)$ and $(C, m'_1)$ induces, via explicit formulas, a minimal model structure $(f,m)$. That is, there are formulas in $(i,p,h)$ and $m'$ that produce the minimal $A_\infty$-algebra structure $m$ on $H^\bullet(C)$ and the quasi-isomorphism $f : (H^\bullet(C), m) \to (C,m')$. 
\end{prop}
\begin{proof}
See \cite[Thm.\ 9.4.14]{LV2012} or \cite[Thm.\ 5.2.5]{A-inf}; both of these references record the formulas. 
\end{proof}
Applying this to the case where $(C,m')$ is a dg-algebra (i.e.\ $m'_n = 0$ for $n \geq 3$) implies Kadeishvili's result on $A_\infty$-algebra minimal models for dg-algebras. 

\begin{rem}
Merkulov set up the same formulas in a more concrete way \cite{merkulov1999}, which the author learned from work of Lu--Palmieri--Wu--Zhang \cite{LPWZ2009}. These formulas may be found in \cite[Ex.\ 5.2.8]{A-inf}, and we give some information here for the reader's convenience. 

Following Merkulov, we note that a homotopy retract between the cohomology $(H,0)$ and the complex it arose from, $(C,d_C)$, amounts to a direct sum decomposition
\begin{equation}
\label{eq: C-decomp}
C^n = B^n \oplus \tilde H^n \oplus L^n \text{ for all } n \geq 0,
\end{equation}
where $B^n$ denotes the subspace of $C^n$ consisting of $n$-coboundaries, $\tilde H^n$ is a complement to $B^n$ in the subspace $Z^n$ of $C^n$ consisting of $n$-cocycles, and $L^n$ is a complement to $Z^n$ in $C^n$. Then $f_1$ in degree $n$ is a map $H^n \ra C^n$ lifting each cohomology class to a choice of representing cocycle. This is specified by the decomposition above as follows: $f_1$ is the inverse of the natural isomorphism $\tilde H^n \risom H^n$. Similarly, $f = (f_n)_{n \geq 1}$ and $m = (m_n)_{n \geq 1}$ are given inductively by formulas in $C^n$ using the decomposition above and the isomorphism $f_1 : H^n \risom \tilde H^n$. 
\end{rem}

It will also be useful to have an inverse quasi-isomorphism to the $f$ of the minimal model structure. 
\begin{prop}
\label{prop: f-inverse}
Let $(C, m')$ and $(i,p,h)$ as in Proposition \ref{prop: KS}, so that we have the minimal model structure $(f,m)$ described there. Then $p$ extends to a quasi-isomorphism of $A_\infty$-algebras, in the following sense: there exists a quasi-isomorphism $g = (g_n)_{n \geq 1}: (C, m') \ra (H^\bullet(C), m)$ such that $g_1 = p$. Moreover, $g$ is a left inverse to $f$, in that $g \circ f : (H^\bullet(C), m) \ra (H^\bullet(C), m)$ is the identity map. That is, $g \circ f$ is an $A_\infty$-isomorphism, where $(g \circ f)_1$ is the identity map $\mathrm{id}_{H^\bullet(C)}$ and $(g \circ f)_n = 0$ for $n \geq 2$. 
\end{prop}

\begin{proof}
This follows from \cite[Thm.\ 3.9(2)]{CL2019}. 
\end{proof}

\subsection{The bar equivalence} 
\label{subsec: bar}

We recall a dualized version of the bar equivalence, which is described at more length in \cite[\S2.1]{A-inf}. 

Let $(A,m)$ be an $A_\infty$-algebra. Taking the suspension of the graded dual of $m_n : A^{\otimes n} \to A$ as described in \S\ref{subsec: conventions}, we get
\[
m_n^*: \Sigma A^* \lra (\Sigma A^*)^{\otimes n}, \text{ of graded degree } 1. 
\]
Taking the product over the codomain, we produce 
\begin{equation}
\label{eq: assemble m*}
m^*  = \prod_{n \geq 1} m_n^* : \Sigma A^* \lra \hat T_k \Sigma A^*. 
\end{equation}
By applying the Leibniz rule, we uniquely extend this map to a derivation
\[
m^* : \hat T_k \Sigma A^* \lra \hat T_k \Sigma A^*. 
\]

Note that nothing in the construction of $m^*$ depends on $m$ satisfying the compatibility conditions demanded of an $A_\infty$-algebra structure on $A$. In fact, $m$ gives an $A_\infty$-algebra structure if and only if the derivation $m^*$ is a differential, i.e.\ $(m^*)^2 = 0$. This is a consequence of the \emph{bar equivalence}, which is an isomorphism (not merely an equivalence) of categories between $A_\infty$-algebras and co-free co-complete co-dg-algebras \cite[pg.\ 7]{proute2011} (see also \cite[\S9.2.1]{LV2012}). The above ``dualized version'' of the bar equivalence restricts to an equivalence on those $A_\infty$-algebras $A$ such that $A^n$ is finite-dimensional for all $n \in \Z$. 

Thus, when $(A,m)$ is an $A_\infty$-algebra, we write 
\[
\B^*(A,m) := (\hat T_k \Sigma A^*, m^*, s)
\]
for the complete dg-algebra given by the differential $m^*$ and the standard multiplication $s$ of $\hat T \Sigma A^*$. In words, we call this the \emph{dual bar construction of $(A,m)$}. 

\subsection{The classical hull}
\label{subsec: classical hull}

There is a natural inclusion functor from $k$-algebras to dg-$k$-algebras, sending a $k$-algebra $D$ to a dg-$k$-algebra $D[0]$ concentrated in degree zero and with a trivial differential. This functor has a left adjoint on dg-$k$-algebras. This functor sends a dg-$k$-algebra $(B, d_B, m_{2,B})$ to its quotient $\cA(B) = \cA(B, d_B, m_{2,B})$ by the ideal generated by 
\[
\bigoplus_{n \in \Z\smallsetminus\{0\}} B^n \text{ and } d_B(B^{-1}),
\]
which we call the \emph{classical hull} of $B$, following \cite[\S2.3]{A-inf}. These generators reflect that $D[0]$ is concentrated in degree $0$ and that $D[0]$ has a trivial differential, respectively. Note that $d_B(B^{-1}) \subset B^0$, because $d_B$ has degree $1$. 

We are especially interested in the case of the dg-algebra $B = \B^*(A,m)$. Because its underlying complete graded algebra is freely generated by $\Sigma A^*$, one may readily compute as in \cite[Ex.\ 2.3.2]{A-inf} that the classical hull is presented as  
\[
\cA(B) = \frac{\hat T_k (\Sigma (A^1))^*}{(m^*((\Sigma (A^2))^*))}. 
\]
Indeed, the degree $-1$ part of the dual bar construction is $\B^*(A,m)^{-1} = (\Sigma (H^2))^*$, while its degree $0$ part is $(\Sigma(H^1))^*$. The projection from $\hat T_k \Sigma H^*$ to $\hat T^k (\Sigma (H^1))^*$ of the image of $m^*(\Sigma (H^2))^*$ can be constructed as in \eqref{eq: assemble m*} from the suspended linear duals of the $A_\infty$-product maps restricted to tensor powers of $H^1$, $m_n : (H^1)^{\otimes n} \to H^2$. 

\subsection{A result from \cite{A-inf}}

Recall from the introduction that $\Omega$ is the Iwasawa algebra of $G$ over $k$ and $\Omega^!$ is the opposite algebra of the Yoneda algebra $\Ext^\bullet_\Omega(k,k)$. The main result that we wish to recall from \cite[\S6]{A-inf} gives a presentation of $\Omega$ in terms of a choice of decomposition of $C^\bullet(\Omega, k)$ as in \eqref{eq: C-decomp}. We state it in terms of its application to $\Omega$. 

\begin{thm}
\label{thm: A-inf main}
Choose a homotopy retract structure on $(H^\bullet(\Omega,k), 0)$ relative to $(C^\bullet(\Omega,k), d_C)$, or, equivalently, a decomposition of $C^\bullet(\Omega,k)$ as in \eqref{eq: C-decomp}. This determines the additional data $(f,m)$ as explained in \S\ref{subsec: MM}.  These data 
determine an isomorphism 
\[
\rho^u: \Omega \lrisom \cA(\B^*(H^\bullet(\Omega,k))) \cong \frac{\hat T_k \Sigma H^1(\Omega,k)^*}{(m^*(\Sigma H^2(\Omega,k)^*)}
\]
given by, for $x \in \Omega$, 
\[
\rho^u: x \mapsto \bar x + \sum_{i =1}^\infty (\underline{e} \mapsto (f_i(\underline{e}))(x)) ,
\]
where $\underline{e}$ is a generic element of $(\Sigma H^1(\Omega,k))^{\otimes i}$ and $x \mapsto \bar x$ denotes reduction modulo the unique maximal ideal of $\Omega$. 
\end{thm}

Let us explain in what sense the map $\underline{e} \mapsto (f_i(\underline{e}))(x)$ denotes an element of $(\Sigma H^1(\Omega,k)^*)^{\otimes i}$, where $i \geq 1$. Notice first that the fixed $f_i$, having graded degree $1-i$, maps $(H^1(\Omega,k))^{\otimes i}$ to $C^1(\Omega,k)$. As $C^1(\Omega,k)$ consists of functions from $\Omega$ to $k$, evaluating $f_i(\underline{e})$ at a fixed choice of $x \in \Omega$ results in the desired map $H^1(\Omega,k)^*)^{\otimes i} \ra k$. 

\begin{proof}[Proof of Theorem \ref{thm: A-inf main}]
The statement of Theorem \ref{thm: A-inf main} is an application of \cite[Cor.\ 6.2.6(1)]{A-inf}, where
\begin{itemize}
\item $\rho$ is the trivial representation $k[G] \to k$, 
\item $\Omega$ replaces the completion $k[G]^\wedge_{\ker \rho}$, 
\item an assumption that $H^n(\Omega,k)$ is finite-dimensional for all $n$ is dropped, since this follows from $G$ being finite-dimensional as a $p$-adic Lie group, and 
\item there are some other simplifications because $\rho : k[G] \ra k$ is the trivial representation in the present case. 
\end{itemize}
Indeed, because $G$ is pro-$p$, the completed group algebra $\Omega$ of $G$ over $k$ is canonically isomorphic to the completion of $k[G]$ at the kernel of the trivial representation $\rho : k[G] \to k$. 
\end{proof}

\begin{rem}
\label{rem: rest of structure}
We see in the presentation of Theorem \ref{thm: A-inf main} what data in the $A_\infty$-algebra $(H^\bullet(\Omega, k), m)$ does not obviously influence the presentation of $\Omega$. Namely, we see that the groups $H^1(\Omega, k)$ and $H^2(\Omega, k)$ along with the $A_\infty$-products $m_n : H^1(\Omega, k)^{\otimes n} \to H^2(\Omega, k)$ determine the isomorphism class of the $k$-algebra $\Omega$. For example, the groups $H^i(\Omega, k)$ for $i \geq 3$ are not obviously involved. Since, conversely, the isomorphism class of $(H^\bullet(\Omega, k), m)$ is determined by $\Omega$, it would be interesting to determine whether and how the entire $A_\infty$-algebra structure $m$ is determined by its restrictions to tensor powers of $H^1(\Omega, k)$ -- namely, $m_n : H^1(\Omega, k)^{\otimes n} \to H^2(\Omega, k)$ -- in the case of a uniform pro-$p$ group $G$, where we completely understand $m_2$ according to \eqref{eq: Lazard theorem}. 
\end{rem}

\section{Answers to Sorensen's questions}
\label{sec: answers}

In this section, we answer the three sub-questions of Question \ref{ques: main}. These answers are, ultimately, applications of Theorem \ref{thm: A-inf main}, which gives a presentation of $\Omega$ in terms of a choice of a homotopy retract between the Hochschild cochain complex $H^\bullet(G,k)$ and the Yoneda algebra $(\Omega^!)^\op$ (which is Hochschild cohomology of the trivial $\Omega$-bimodule $k$). 

The main obstacle in the way of directly addressing Question \ref{ques: main} using Theorem \ref{thm: A-inf main} is that Question \ref{ques: main} and Theorem \ref{thm: A-inf main} address $\Omega$ in terms of two \emph{isomorphic yet different} $A_\infty$-algebra structures that enrich the Yoneda graded algebra $(\Omega^!)^\op$. This obstacle is overcome by applying a result of Segal, recorded here as Proposition \ref{prop: segal map}. The main result we prove is Corollary \ref{cor: Yoneda presentation} 

For clarity, we highlight the notation that we will use in this section for these two isomorphic but different choices $m$ and $m'$ of $A_\infty$-algebra structure enriching $\Omega^!$. Importantly, we will go on to explain in this section that a single choice induces both of $m$ and $m'$ as well as the structure that reconciles them. 
\begin{notn} \ 
\begin{itemize}
\item We let $m$ denote an $A_\infty$-algebra structure enriching $\Omega^!$ that arises from transfer of structure from $\cH_G^\bullet$, as appeared in Sorensen's work \cite{sorensen2020} as overviewed in \S\ref{subsec: intro equivalences} and \S\ref{sec: sorensen}.
\item We let $m'$ denote an $A_\infty$-algebra structure enriching $H^\bullet(G,k) \cong (\Omega^!)^\op$ that arises from transfer of structure from $C^\bullet(G,k)$, as appeared in the author's work \cite{A-inf} as overviewed in \S\ref{sec: recall}. 
\end{itemize}
\end{notn}

\subsection{Compatibility of the two endomorphism dg-algebras}

We begin with a definition of compatibility of a homotopy retract between complexes. 

\begin{defn}
Choose a homotopy retract $(i',p',h')$ (resp.\ $(i,p,h)$) between a cochain complex $C'$ (resp.\ $C$) and its cohomology $H' = H^\bullet(C')$ (resp.\ $H = H^\bullet(C)$). Let $\Psi : C' \ra C$ be a quasi-isomorphism, which therefore induces an isomorphism $H^\bullet(\Psi) : H' \risom H$. We say that $\Psi$ is \emph{compatible} with these two homotopy retracts when we have commutative squares 
\[
\xymatrix{
H' \ar[r]^{i'} \ar[d]^{H^\bullet(\Psi)} & C' \ar[d]^\Psi & C' \ar[r]^{p'} \ar[d]^\Psi & H' \ar[d]^{H^\bullet(\Psi)}& \Sigma C' \ar[r]^{h'} \ar[d]^{\Sigma \Psi} & C' \ar[d]^\Psi  \\
H \ar[r]^{i} & C &  C \ar[r]^{p} & H & \Sigma C \ar[r]^{h} & C
}
\]
\end{defn}

\begin{rem}
This may be too narrow of a definition of a compatible homotopy retract for a general theory, but it suffices for the situation at hand. 
\end{rem}

We will use the following pair of compatible homotopy retracts, given a quasi-isomorphism and extra maps. 
\begin{lem}
\label{lem: compat}
Let $\Psi$ be a quasi-isomorphism of complexes $\Psi : C' \ra C$ that admits a left inverse quasi-isomorphism $\Phi : C \ra C'$, i.e.\ $\Phi \circ \Psi = \mathrm{id}_{C'}$. Let $(i, p, h)$ be a homotopy retract between $C$ and $H$. Then the following natural formulas produce a compatible homotopy retract $(i',p',h')$ between $C'$ and $H'$.
\[
i' = \Phi \circ i \circ H^\bullet(\Psi), \qquad 
p' = H^\bullet(\Phi) \circ p \circ \Psi, \qquad 
h' = \Phi \circ h \circ \Sigma \Psi.
\]
\end{lem}

\begin{proof}
The homotopy retract property of $(i',p',h')$ follows by direct computation. The compatibility follows directly from the formulas. 
\end{proof}

\begin{defn}
Let $E^\bullet(\Omega, k)$ denote the ($k$-linear) endomorphism dg-algebra of the projective bar resolution of $k$, whose terms consist of left $\Omega$-modules
\[
E^{-i}(\Omega,k) = \left\{\begin{array}{ll}\Omega^{\hat \otimes (i+1)} & i \geq 0 \\ 0 & i < 0,
\end{array}\right.
\]
with differentials described in \cite[bottom of pg.\ 162]{sorensen2020}. 
\end{defn}
Observe that $E^\bullet(\Omega,k)$ is canonically isomorphic to the opposite algebra of the dg-Hecke algebra $\cH_G^\bullet$, and therefore there is a designated identification $H^\bullet(E^\bullet(\Omega,k)) \cong (\Omega^!)^\op$. As usual, $C^\bullet(\Omega, k)$ denotes the dg-algebra of Hochschild cochains of the $\Omega$-bimodule $k$. 

\begin{prop}
\label{prop: segal map}
There exists an explicit quasi-isomorphism of dg-algebras $\Psi: C^\bullet(\Omega, k) \ra E^\bullet(\Omega, k)$ that admits a (non-multiplicative) left inverse of cochain complexes $\Phi : E^\bullet(\Omega, k) \ra C^\bullet(\Omega, k)$. Moreover, the isomorphism of graded algebras 
\[
H^\bullet(\Psi) : H^\bullet(\Omega, k) := H^\bullet(C^\bullet(\Omega, k)) \lrisom H^\bullet(E^\bullet(\Omega, k)) \cong (\Omega^!)^\op 
\]
equals the canonical isomorphism recorded in \eqref{eq: Yoneda algebra}. 
\end{prop}

\begin{proof}
This is the content of \cite[Lem.\ 2.6]{segal2008}. 
\end{proof}

The following corollary sums up the relationship between the two $A_\infty$-algebra structures on the Yoneda algebra $(\Omega^!)$ that we have seen: $m^\op$ arises from transfer of structure from $E^\bullet(\Omega,k)$, while $m'$ arises from $C^\bullet(\Omega,k)$. 

\begin{cor}
\label{cor: str transfer}
Let $\Psi$, $\Phi$, $C^\bullet(\Omega, k)$, $E^\bullet(\Omega,k)$ be as in Proposition \ref{prop: segal map}. Choose a homotopy retract $(i, p, h)$ between $E^\bullet(\Omega, k)$ and $(\Omega^!)^\op$. These choices produce 
\begin{itemize}
\item  a compatible homotopy retract $(i',p',h')$ of $C^\bullet(\Omega,k)$ determined in Lemma \ref{lem: compat}. 
\item $(f', m')$ (resp.\ $(f^\op, m^\op)$), which is the minimal model structure on $H^\bullet(G,k)$ (resp.\ $(\Omega^!)^\op$) induced by $(i',p',h')$ (resp.\ $(i,p,h)$) according to the formulas of Proposition \ref{prop: KS}. 
\item $g^\op : E^\bullet(\Omega, k) \ra ((\Omega^!)^\op, m^\op)$, which is the left inverse to $f^\op$ determined in Proposition \ref{prop: f-inverse}. 
\end{itemize}
In addition,  the isomorphism of graded algebras $H^\bullet(\Psi) : H^\bullet(\Omega, k) \risom (\Omega^!)^\op$ extends to an isomorphism of $A_\infty$-algebras determined by 
\[
\Upsilon: (H^\bullet(\Omega, k), m') \lrisom ((\Omega^!)^\op, m^\op),
\]
given by 
\[
\Upsilon = g^\op \circ \Psi \circ f'
\]
\end{cor}

\begin{proof}
All that we need to check is that the given $A_\infty$-isomorphism $\Upsilon$ extends $H^\bullet(\Psi)$; that is, $\Upsilon_1 = H^\bullet(\Psi)$. This follows from the last sentence of Proposition \ref{prop: f-inverse}. 
\end{proof}

Now we combine the foregoing corollary with the presentation of $\Omega$ in terms of cohomological data given in Theorem \ref{thm: A-inf main}, yielding the main result. 

\begin{cor}
\label{cor: Yoneda presentation}
Choose a homotopy retract $(i', p', h')$ between $E^\bullet(\Omega, k)$ and $(\Omega^!)^\op$. This choice induces a presentation of $\Omega$ in terms of $(\Omega^!, m)$ and other data induced by $(i', p', h')$ enumerated in Corollary \ref{cor: str transfer}. The presentation is given by
\begin{align*}
\Omega \lrisom \cA(\B^*((\Omega^!)^\op, m^\op)) = \frac{\hat T_k \Sigma ((\Omega^!)^1)^*}{(m^{\op*}(\Sigma ((\Omega^!)^2)^*)} \\
\rho^u: x \mapsto \bar x + \sum_{i =1}^\infty (\underline{e} \mapsto ((f \circ \Upsilon^{-1})_i(\underline{e}))(x)),
\end{align*}
where $\underline{e}$ is a generic element of $(\Sigma (\Omega^!)^1)^{\otimes i}$ and $x \mapsto \bar x$ denotes reduction modulo the unique maximal ideal of $\Omega$. 
\end{cor}

The meaning of $(\underline{e} \mapsto ((f \circ \Upsilon^{-1})_i(\underline{e}))(x))$ is just as explained after Theorem \ref{thm: A-inf main}, keeping in mind that $\Upsilon^{-1}$ is an $A_\infty$-isomorphism $((\Omega^!)^\op, m^\op) \risom (H^\bullet(\Omega, k), m')$. 

\begin{proof}
This is a combination of Corollary \ref{cor: str transfer} and Theorem \ref{thm: A-inf main}. 
\end{proof}

\subsection{Question (a): Characterizing the Iwasawa algebra with $A_\infty$-products}

We prove that the $A_\infty$-enrichment of the Yoneda algebra of $\Omega$ characterizes $\Omega$ up to isomorphism. 
\begin{thm}
\label{thm: a}
The isomorphism class of the $A_\infty$-algebra $(\Omega^!, m)$ determines $\Omega$ up to isomorphism. 
\end{thm}

\begin{proof}
A homotopy retract between $(\Omega^!)^\op$ and $E^\bullet(\Omega, k)$ induces an $A_\infty$-structure $m^\op$ enriching $(\Omega^!)^\op$, and the isomorphism class of $((\Omega^!)^\op,m^\op)$ is unique up to isomorphism -- this is simply the result of Kadeishvili (\S\ref{subsec: MM}). By Corollary \ref{cor: str transfer}, $((\Omega^!)^\op,m^\op)$ and $(H^\bullet(\Omega,k),m')$ are in the same isomorphism class. Since the passage from an $A_\infty$-algebra to the classical hull of its dual bar construction sends isomorphisms to isomorphisms, the theorem follows.
\end{proof}

\subsection{Question (b): Trivial $A_\infty$-structures and abelianness}

A minimal $A_\infty$-structure $m$ is called \emph{trivial} when $m_n = 0$ for $n \geq 3$. 

\begin{thm}
\label{thm: b}
The $A_\infty$-algebra structure $m$ enriching $\Omega^!$ is trivial if and only if $G$ is abelian. In that case, $\Omega \simeq k \lb x_1, \dotsc, x_d\rb$, where $d$ is the $k$-dimension of $\Ext^1_\Omega(k,k)$. 
\end{thm}

\begin{proof}
Sorensen proved the ``only if'' implication in \cite[Thm.\ 1.2]{sorensen2020} and asked about the converse. Because $G$ injects into the units of the completed group algebra $\Omega$, it suffices to prove that $\Omega$ is commutative. 

By \cite[Thm.\ 1.2]{sorensen2020}, the $A_\infty$-algebra structure $m'$ enriching the Yoneda algebra $(\Omega')^!$ of $\Omega' := \Omega(\Z_p^d)$ is trivial. Therefore, when $(\Omega^!, m)$ is trivial, it is $A_\infty$-isomorphic to $((\Omega')^!,m')$. Then, by Theorem \ref{thm: a}, there exists an isomorphism $\Omega \simeq \Omega(\Z_p^d)$, so $\Omega$ is commutative. 
\end{proof}

\begin{rem}
The author thanks Claus Sorensen for suggesting the efficient proof above upon seeing an earlier version of this paper. For the purpose of illustrating what calculations lie below the result, the following more explicit argument still may be instructive as to the role of the $A_\infty$-structure and its triviality. 
\end{rem}

\begin{proof}[Alternate proof of Theorem \ref{thm: b}]
By Corollary \ref{cor: Yoneda presentation}, we have a presentation for $\Omega$ in terms of $\cA(\B^*((\Omega^!)^\op,m^\op))$, where $m^\op_n = 0$ for $n=1$ or $n \geq 3$, and $m_2^\op$ is given by the isomorphism $\Omega^! \cong \bigwedge \frg^*$ of \eqref{eq: Lazard theorem}. Recall from \S\ref{subsec: bar} that the expression $m^{\op *}$ determining $\cA(\B^*((\Omega^!)^\op,m^\op))$ in Corollary \ref{cor: Yoneda presentation} is the product over $n$ of the suspended linear duals $m_n^{\op*} : \Sigma \Ext^2_\Omega(k,k)^* \ra  (\Sigma \Ext^1_\Omega(k,k)^*)^{\otimes n}$ of the $A_\infty$-structure maps $m_n^\op : \Ext^1_\Omega(k,k)^{\otimes n} \ra \Ext^2_\Omega(k,k)$. Thus we are only concerned with the degree 2 contribution 
\[
m_2^{\op*} : \Sigma\Ext^2_\Omega(k,k)^* \ra (\Sigma\Ext^1_\Omega(k,k)^*)^{\otimes 2}. 
\]

To calculate $m_2^{\op *}$, we note that the isomorphism $\Omega^! \cong \bigwedge \frg^*$ supplies a canonical isomorphism
\[
\wedge^2 \Ext^1_\Omega(k,k) \lrisom \Ext^2_\Omega(k,k)
\]
such that the multiplication $m_2^\mathrm{op}$ in the graded algebra $(\Omega^!)^\mathrm{op}$, restricted to $\Ext^1_\Omega(k,k)$, is the composition of this map with the standard projection 
\[
\Ext^1_\Omega(k,k)^{\otimes 2} \rsurj \wedge^2 \Ext^1_\Omega(k,k). 
\]
Therefore, the image of $m_2^{\op *}$ in $(\Sigma\Ext^1_\Omega(k,k)^*)^{\otimes 2}$ is precisely the alternating tensor subspace. This completes this alternate proof of Theorem \ref{thm: b}. 
\end{proof}

\subsection{Question (c): Change of group}

In this section, we work with an open subgroup $K \subset G$. Correspondingly, we write $\Omega(G), \Omega(K)$ for their Iwasawa algebras. And for all objects discussed in previous sections with respect to $G$, we label them with a subscript $G$ or $K$ to indicate which group they are associated with, e.g.\ $f_G$ and $f_K$.

\begin{thm}
\label{thm: c}
Let $K \subset G$ be an open subgroup. Then there is a morphism of $A_\infty$-algebras $(\Omega(G)^!, m_G) \ra (\Omega(K)^!, m_K)$ compatible with the restriction map $\mathrm{Mod}(\Omega(G)) \ra \mathrm{Mod}(\Omega(K))$, where ``compatible'' means that we have associated this map of $A_\infty$-algebras to the right-hand vertical arrow in this diagram 
\begin{equation}
\label{eq: c}
\xymatrix{
\Omega(K) \ar[d] \ar[rr]^(.35)\sim & & \frac{\hat T_k \Sigma \Ext^1_{\Omega(K)}(k,k)^*}{(m_K^{\op *}((\Sigma \Ext^2_{\Omega(K)}(k,k)^*))} \ar[d] \\
\Omega(G) \ar[rr]^(.35)\sim & & \frac{\hat T_k \Sigma \Ext^1_{\Omega(G)}(k,k)^*}{(m_G^{\op *}((\Sigma \Ext^2_{\Omega(G)}(k,k)^*)))}
}
\end{equation}
where the horizontal presentation maps arise from Corollary \ref{cor: Yoneda presentation} and the diagram commutes up to inner automorphism in $\Omega(K)$. 
\end{thm}
Because the diagram commutes up to inner automorphism, it induces the map of module categories required by sub-question (c) of Question \ref{ques: main}. 

The proof relies upon using Hochschild cohomology to produce the diagram above, and then applying the isomorphism $\Upsilon$ of Corollary \ref{cor: str transfer} at the end.

\begin{proof}
Choose two (independent) homotopy retracts as in Corollary \ref{cor: str transfer}, one for objects associated to $G$, and one for objects associated to $K$. This results in the objects enumerated there ($f, \Upsilon$, etc.), which we will now use with a subscript to indicate whether they are associated with $G$ or $K$ ($f_K, \Upsilon_K$, etc.). In addition, we require a left inverse $g_K$ of $f_K$, as in Proposition \ref{prop: f-inverse}. 

We link the objects associated to $G$ to those associated to $K$ by via the natural map of Hochschild cochains $C^\bullet(G, k) \ra C^\bullet(K,k)$ induced by restricting functions of $G^{\times i}$ to its subgroup $K^{\times i}$. There is a morphism of $A_\infty$-algebras $(H^\bullet(G,k), m_G') \ra H^\bullet(K,k), m_K')$ resulting from the composite 
\begin{equation}
\label{eq: G to K}
\eta_H: H^\bullet(G,k) \buildrel{f_G}\over\lra C^\bullet(G,k) \buildrel{\mathrm{restr.}}\over\lra C^\bullet(K,k) \buildrel{g_K}\over\lra H^\bullet(K,k). 
\end{equation}
Next, define $\eta : ((\Omega(G)^!)^\op, m_G^\op) \to ((\Omega(K)^!)^\op, m_K^\op)$ by $\eta := \Upsilon_K \circ \eta_H \circ \Upsilon^{-1}_G$. This $\eta$ is the opposite $A_\infty$-morphism to the desired morphism in the statement of the theorem. 

This morphism $\eta$ is compatible with the natural restriction map of quasi-compact module categories $\mathrm{Mod}(\Omega(G)) \ra \mathrm{Mod}(\Omega(K))$ because -- we claim -- \eqref{eq: c} commutes up to inner automorphism by the domain $\Omega(K)$. This claim of commutativity follows from the following facts, where we use the word ``corresponds'' to indicate an contravariant matching of maps between
\begin{itemize}
\item on one hand, $A_\infty$-algebras (which include dg-algebras), and 
\item other the other hand, maps among the Iwasawa algebras (which play the role of a dual to $C^\bullet(G,k)$) and classical hulls of the dual-bar constructions (which play the role of a dual to minimal $A_\infty$-algebras). 
\end{itemize}

The right-hand downward map in \eqref{eq: c} corresponds to $\eta$, while the $A_\infty$-quasi-isomorphisms
\[
f_G \circ \Upsilon^{-1}_G: (\Omega(G)^!)^\op \to C^\bullet(G,k), \quad f_K \circ \Upsilon^{-1}_K: (\Omega(K)^!)^\op \to C^\bullet(K,k)
\]
correspond to the presentations appearing as the horizontal pair of arrows in the theorem statement, via the formula of Theorem \ref{thm: A-inf main}.

Thus the clockwise map in \eqref{eq: c} corresponds to 
\[
f_K \circ \Upsilon_K^{-1} \circ \eta : (\Omega(G)^!)^\op \to C^\bullet(K,k),
\]
while the counter-clockwise map in \eqref{eq: c} corresponds to the composition of the leftmost two maps of \eqref{eq: G to K}, which we now denote by $t := (\mathrm{restr.}) \circ f_G \circ \Upsilon_G^{-1}$. Re-expressing the $A_\infty$-morphisms corresponding to the clockwise and counter-clockwise maps in terms of $t$, one observes that they correspond to $f_K \circ g_K \circ t$ and $t$, respectively. 

From \cite[Thm.\ 6.2.3]{A-inf} and the discussion following it, we know that $f_K \circ g_K$ corresponds to an inner automorphism of $\Omega(K)$, from which the theorem follows. 
\end{proof}

\section*{Acknowledgements}

The author would like to thank Claus Sorensen for his interest in \cite{A-inf} and for stimulating correspondence. The author also would like to thank Karol Koziol, Leonid Positselski, and Peter Schneider for helpful and clarifying correspondence about previous versions of this paper. The author was partially supported by Engineering and Physical Sciences Research Council grant EP/L025485/1 and partially supported by Simons Foundation award 846912.

\bibliographystyle{alpha}
\bibliography{CWEbib-2023-CIS}

\end{document}